\documentclass[a4paper,
              12pt,
              twoside
              ]
              {article}

\usepackage{latexsym, amssymb}
\usepackage{amsmath,amsthm,mathtools,amssymb,eqnarray}

\usepackage[utf8]{inputenc}
\usepackage[english]{babel}
\usepackage{csquotes}

\usepackage[
dashed=false,
natbib=true,
sorting=nyt,%
style=ieee,%
language=british,%
citestyle=numeric-comp,%
backend=bibtex,
backref=true,%
hyperref=true,%
maxcitenames=8,%
maxbibnames=100,%
block=none,%
url=false,
isbn=false,
doi=false,
defernumbers=true]{biblatex}
\DefineBibliographyStrings{english}{%
  backrefpage  = {cited on page}, 
  backrefpages = {cited on pages} 
}

\usepackage{bookmark}
\bookmarksetup{
  numbered,
  open,
}

\usepackage{geometry}

\usepackage{fancyhdr}

\usepackage[utf8]{inputenc}   

\usepackage{enumerate}      

\usepackage{booktabs}

\title{Sidon sets, sum-free sets and linear codes}
\author{
  Ingo Czerwinski\thanks{Faculty of Mathematics,
        Otto von Guericke University Magdeburg,
        39106 Magdeburg, Germany,
        (\texttt{ingo@czerwinski.eu}, \texttt{alexander.pott@ovgu.de})} \;
  and Alexander Pott\footnotemark[1]
}
\date{}

\newcommand{\moreSpaceBetweenRows}[1]{\renewcommand{\arraystretch}{#1}}

\newtheorem{thm}{Theorem}[section]
\newtheorem{cor}[thm]{Corollary}
\newtheorem{lem}[thm]{Lemma}
\newtheorem{prop}[thm]{Proposition}
\newtheorem{defn}[thm]{Definition}
\newtheorem{rem}[thm]{Remark}

\newcommand{\secref}[1]{Section~\ref{#1}}

\newcommand{\thmref}[1]{Theorem~\ref{#1}}
\newcommand{\corref}[1]{Corollary~\ref{#1}}
\newcommand{\lemref}[1]{Lemma~\ref{#1}}
\newcommand{\propref}[1]{Proposition~\ref{#1}}

\newcommand{\tabref}[1]{Table~\ref{#1}}

\newcommand{\nat}{\mathbb{N}}

\newcommand{\inte}{\mathbb{Z}}

\newcommand{\binf}{\mathbb{F}_2}

\newcommand{\binvn}[1]{\mathbb{F}_2^{#1}}

\newcommand{\eps}{\varepsilon}

\newcommand{\dotcup}{\mathbin{\mathaccent\cdot\cup}}

\providecommand{\transpose}[1]{#1^\intercal}

\newcommand{\wider}[1]{\;#1\;}

\newcommand{\set}[1]{\lbrace #1 \rbrace}
\newcommand{\sett}[2]{\lbrace #1 : #2 \rbrace}
\newcommand{\setminzero}{\setminus\set{0}}

\providecommand{\vspan}[1]{\langle #1\rangle}

\providecommand{\kSums}[2]{\mathcal{S}_{#1}(#2)}

\providecommand{\twoSums}[1]{\kSums{2}{#1}}

\providecommand{\threeSums}[1]{\kSums{3}{#1}}

\providecommand{\fourSums}[1]{\kSums{4}{#1}}

\providecommand{\kStarSums}[2]{\mathcal{S}_{#1}^*(#2)}
\providecommand{\kStarSumsB}[2]{\mathcal{S}_{#1}^*\bigl(#2\bigr)}

\providecommand{\twoStarSums}[1]{\kStarSums{2}{#1}}
\providecommand{\twoStarSumsB}[1]{\kStarSumsB{2}{#1}}

\providecommand{\threeStarSums}[1]{\kStarSums{3}{#1}}
\providecommand{\threeStarSumsB}[1]{\kStarSumsB{3}{#1}}

\providecommand{\fourStarSums}[1]{\kStarSums{4}{#1}}

\providecommand{\code}{\mathcal{C}}
\providecommand{\checkMat}{\mathcal{H}}
\providecommand{\assoMat}{\mathcal{M}}
\providecommand{\punct}[1]{\mathcal{PU}(#1)}
\providecommand{\dmax}[1]{d_{max}(#1)}
\providecommand{\dmaxTab}{(\dmax{n, k})_{n,k}}
\providecommand{\orth}{\perp}

\providecommand{\smax}[1]{s_{max}(\binvn{#1})}
\DeclareMathOperator{\sfs}{\mathit{sfs}}
\providecommand{\sfsmax}[1]{\sfs_{max}(\binvn{#1})}

\providecommand{\abs}[1]{\left\lvert#1\right\rvert}

\providecommand{\floor}[1]{\left\lfloor#1\right\rfloor}

\providecommand{\mathKeyword}[1]{\textit{#1}}

\providecommand{\keywords}[1]{\textbf{Keywords} #1 \\[12px]}
\providecommand{\subclass}[1]{\textbf{Mathematics Subject Classification (2020)} #1 }

\providecommand{\equRefTrivialBound}{(1)}
\providecommand{\equRefNewBound}{(2)}
\providecommand{\equRefContradict}{(3)}
\providecommand{\equRefEstimate}{(4)}
\providecommand{\equRefAoddBtwo}{(5)}
\providecommand{\equRefAoddBone}{(6)}
\providecommand{\equRefAoddBzero}{(7)}
\providecommand{\equRefAevenBtwo}{(8)}
\providecommand{\equRefAevenBone}{(9)}

\begin{document}

\maketitle

\renewcommand{\sectionmark}[1]{}

\renewcommand{\labelenumi}{(\alph{enumi})} 

\begin{abstract}
  Finding the maximum size of a Sidon set in $\binvn{t}$ is of research interest
  for more than  40 years. In order to tackle this problem
  we recall a  one-to-one correspondence between sum-free Sidon sets
  and linear codes with minimum distance greater or equal 5.
  Our main contribution about codes is a new non-existence result
  for linear codes with minimum distance 5 based on a sharpening
  of the Johnson bound.
  This gives, on the Sidon set side, an improvement
  of the general upper bound for the maximum size of a Sidon set.
  Additionally, we characterise maximal Sidon sets, that are those Sidon sets
  which can not be extended by adding elements without loosing the Sidon property,
  up to dimension 6 and give all possible sizes for dimension 7 and 8
  determined by computer calculations.
\end{abstract}


\subclass{11B13, 94B05, 94B65}



\section{Introduction}

In the early 1930s Sidon introduced $B_2$-sequences of positive integers in
connection with his work on Fourier analysis \cite{Sidon32}, \cite{Sidon35}.
Later, Babai and S{\'o}s \cite{BS85} generalised
the definition of $B_2$-sequences to arbitrary groups and called them Sidon sets.
In this work, we focus only on Sidon sets in $\binvn{t}$, the $t$-dimensional vector space over the binary field $\binf$.

\begin{defn}
  Let $M$ be a subset of $\binvn{t}$. $M$ is called \mathKeyword{Sidon} if $m_1+m_2\ne m_3+m_4$ for all pairwise distinct $m_1,m_2,m_3,m_4\in M$.
\end{defn}

Since the definition of Sidon sets is based on sums, we introduce the following notation:
Let $M$ be a subset of $\binvn{t}$. For any $k\geq 2$ we call
\[
\kSums{k}{M} = \sett{m_1 + \dotsb + m_k}
{m_1, \dotsc, m_k \in M }
\]
the \mathKeyword{$k${-}sums} of $M$ and
\[
\kStarSums{k}{M} = \sett{m_1 + \dotsb + m_k}
{m_1, \dotsc, m_k \in M \text{ pairwise distinct} }
\]
the \mathKeyword{$k${-}star{-}sums} of $M$.
In this paper, only 2, 3 and 4-(star)-sums are considered.
The characteristic 2 of $\binvn{t}$ leads to the following
frequently used properties:
{\it
  \begin{enumerate}
    \item $\twoSums{M} = \twoStarSums{M} \cup \set{0}$;
    \item $\threeSums{M} = \threeStarSums{M} \cup M$;
    \item $\fourSums{M} = \fourStarSums{M} \cup \twoStarSums{M} \cup \set{0}$.
  \end{enumerate}
}

We recall also, that the \mathKeyword{(Hamming) weight} of a vector
is the number of its non-zero entries.
The weight of a vector is of interest when considering the sums of the elements of a set containing the standard basis.

The Sidon property of $M$ can be characterised in terms
of $2${-}star{-}sums and $3${-}star{-}sums
by the following equivalent statements:
{\it
  \begin{enumerate}
    \item $M$ is Sidon;
    \item $\abs{\twoStarSums{M}} = \binom{\abs{M}}{2}$;
    \item $\threeStarSums{M} \cap M = \emptyset$.
  \end{enumerate}
}

Now that we have briefly introduced Sidon sets,
in \secref{sec:maximal-sidon-sets} we will begin to discuss
the fundamental problem of their maximum size.
For this we study maximal Sidon sets, which are Sidon sets
that cannot be extended by adding new elements without losing the Sidon property.
\secref{sec:sum-free-sets} introduces sum-free sets
which are used to recall a one-to-one correspondence
between additive structures and linear codes in \secref{sec:linear-codes}.
In \secref{sec:non-existence-results} we give a new non-existence result for linear codes
with minimum distance 5 based on a sharpening of the Johnson bound,
which, on the Sidon set side, gives an improvement of the general upper bound
for the maximum size of a Sidon set
(\thmref{thm:improved-codes-non-existence-sidon-bound}).

We note that several statements in \secref{sec:maximal-sidon-sets}
and \ref{sec:linear-codes}
have also been investigated in connection with almost perfect nonlinear (APN)
functions, which are special types of Sidon sets.
For instance, \propref{prop:sidon-extension} and \ref{prop:max-sidon-charac}
can be found, in APN language, in \cite{carlet22},
our \propref{prop:sumfree-sidon-code} is related to Theorem 5 in \cite{CCZ98},
and \thmref{thm:large-sidon-set-properties} is related to
Proposition 4 in \cite{CCZ98}.

Another combinatorial problem and its connections to linear codes,
which contains Sidon sets as a special case, is studied in \cite{Sidorenko2020}.

\section{Maximal Sidon sets}
\label{sec:maximal-sidon-sets}

The fundamental problem about a Sidon set
is the question about its maximum size, which was already discussed
about 40 years ago by Babai and S{\'o}s \cite{BS85}.
We will denote by $\smax{t}$ the maximum size of a Sidon set in $\binvn{t}$.
An upper bound arises directly from the fact that all
2{-}star{-}sums of $M$ have to be distinct and non-zero, hence
\[
\binom{\abs{M}}{2} = \frac{\abs{M}(\abs{M}-1)}{2}
\leq \abs{\binvn{t}\setminzero}.
\]

This bound is still the best known upper bound and
was translated by Carlet and Mesnager \cite{CM22} into an explicit form:
\[
\smax{t} \leq \floor{\frac{1 + \sqrt{2^{t+3}-7}}{2}}.
\]

In the next Proposition we rewrite this bound slightly and
call it from now on the \mathKeyword{trivial upper bound} for
the maximum size of a Sidon set.
Later in \thmref{thm:improved-codes-non-existence-sidon-bound}
we will improve this trivial upper bound.

\begin{prop} \label{prop:max-sidon-upper-bound}
  For any $t$, an upper bound for the maximum size of a Sidon set
  in $\binvn{t}$ is given by
  \begin{equation} 
  \smax{t} \leq
  \begin{cases}
    2^{\frac{t + 1}{2}} &\text{ for $t$ odd},\\
    \floor{\sqrt{2^{t+1}} + 0.5}     &\text{ for $t$ even}.
  \end{cases}
  \end{equation}
\end{prop}
\begin{proof}
  Let $M\subseteq \binvn{t}$ be a Sidon set. Then the sums $m_1+m_2$ are
  distinct and non-zero for all distinct $m_1,m_2\in M$ by the Sidon definition. Therefore
  \[
  \binom{\abs{M}}{2} = \frac{\abs{M}(\abs{M}-1)}{2}
  \leq \abs{\binvn{t}\setminzero}.
  \]
  When $t$ is odd, then $\abs{M} = 2^{\frac{t + 1}{2}}$ fulfills
  this inequality, but $\abs{M} = 2^{\frac{t + 1}{2}}+1$ not anymore.

  Let $t$ be even. We show that
  $\floor{\frac{1 + \sqrt{2^{t+3}-7}}{2}} = \floor{\sqrt{2^{t+1}} + 0.5}$.
  Of course, we have
  $\frac{1 + \sqrt{2^{t+3}-7}}{2} \wider{\leq} \sqrt{2^{t+1}} + 0.5$.
  Assume that there exists $k\in\nat$ such that
  \[
  \frac{1 + \sqrt{2^{t+3}-7}}{2} \wider{<} k \wider{\leq} \sqrt{2^{t+1}} + 0.5,
  \]
  which is equivalent to
  \[
  2^{t+1} - \frac{7}{4} \wider{<} (k - \frac{1}{2})^2 \wider{\leq} 2^{t+1}
  \]
  and
  \[
  2^{t+1} - 2 \wider{<} k(k-1) \wider{\leq} 2^{t+1} - \frac{1}{4}.
  \]
  But this contradicts $k(k-1)$ even.
\end{proof}

After looking at an upper bound for the maximum size of a Sidon set,
it is natural to ask whether every Sidon set can be extended to a
Sidon set of maximum size by adding elements.
Already in dimension 6 this is not true, as we will see later in
\propref{prop:max-sidon-charac-small-dim}.

Therefore, we now concentrate on the question how a Sidon set
can be extended by adding elements without losing the Sidon property.
Before we state the next Proposition from \cite{RRW22},
recall the following notation for sets $A$, $B$ and $C$:
$A\dotcup B = C$ means $A\cup B = C$ and $A\cap B=\emptyset$.

\begin{prop}[\cite{RRW22}]\label{prop:sidon-extension}
  Let $M$ be a Sidon set in $\binvn{t}$ and $g\in\binvn{t}\setminus M$.
  Then $M\dotcup\set{g}$ is Sidon if and only if
  $g\in\binvn{t}\setminus\threeSums{M} =
  \binvn{t}\setminus\big(\threeStarSums{M} \dotcup M\big)$.
\end{prop}
\begin{proof}
Let $M\subseteq \binvn{t}$ be a Sidon set.
From $g\in\binvn{t}\setminus M$ follows that
$M\dotcup\set{g}$ is not Sidon if and only if there exist $m,m_1,m_2\in M$
pairwise distinct such that $g+m=m_1+m_2$ which is equivalent to $g=m+m_1+m_2$.
Hence
$g\in\threeStarSums{M}\subseteq \threeStarSums{M} \dotcup M = \threeSums{M}$.
\end{proof}

Those Sidon sets which can not be extended by adding elements
without losing the Sidon property are of particular interest:

\begin{defn}
  A Sidon set $M\subseteq\binvn{t}$ is called \mathKeyword{maximal}
  if $M = S$ for every Sidon set $S$ with $M \subseteq S \subseteq \binvn{t}$.
\end{defn}

\propref{prop:sidon-extension} helps us to characterise maximal Sidon sets via
their $3${-}star{-}sums and $3${-}sums.

\begin{prop}[\cite{RRW22}]\label{prop:max-sidon-charac}
  Let $M$ be a Sidon set in $\binvn{t}$.
  Then the following statements are equivalent:
  \begin{enumerate}
    \item $M$ is maximal;
    \item $\threeSums{M} = \binvn{t}$;
    \item $\threeStarSums{M} \dotcup M = \binvn{t}$
    (that means: $\threeStarSums{M} \cup M = \binvn{t}$ and
                 $\threeStarSums{M} \cap M = \emptyset$).
  \end{enumerate}
\end{prop}

The property of being Sidon and of being maximal Sidon is invariant
under the action of the affine group.

\begin{prop}\label{prop:Sidon-equiv}
  Let $M$ be a subset of $\binvn{t}$ and $T\colon\binvn{t}\to \binvn{t}$ be
  an affine permutation. Then
  \begin{enumerate}
    \item $M$ is Sidon if and only if $T(M)$ is Sidon; and
    \item $M$ is maximal Sidon if and only if $T(M)$ is maximal Sidon.
  \end{enumerate}
\end{prop}
\begin{proof}
  Let be $T=L+a$ where $L\colon\binvn{t}\to \binvn{t}$
  is a a linear permutation and $a\in\binvn{t}$.
  From $T$ affine and bijective follows $\abs{T(M)} = \abs{M}$,
  \begin{align*}
    \abs{\twoStarSumsB{T(M)}} &= \abs{L\bigl(\twoStarSums{M}\bigr)} = \abs{\twoStarSums{M}} = \binom{\abs{M}}{2},\\
    \abs{\threeStarSumsB{T(M)}} &= \abs{T\bigl(\threeStarSums{M}\bigr)} = \abs{\threeStarSums{M}} = \abs{\binvn{t}\setminus M},
  \end{align*}
  and therefore (a) and (b).
\end{proof}

After introducing maximal Sidon sets we shortly mention the well-known fact
that the graph of an APN function is a Sidon set.
It \cite{RRW22} it is shown that the graph of the classical APN example $x^3$
is a maximal Sidon set.
And in \cite{carlet22}, APN functions whose graphs are maximal Sidon sets are discussed in general.
Recently, some maximal Sidon sets (first introduced in \cite{CM22}, \cite{CP21}),
which are larger than graphs of APN functions, are discussed in \cite{nagy22}.

In order to characterise maximal Sidon sets in small dimensions,
we recall some basic properties of subsets of $\binvn{t}$.
Here $\vspan{M}$ denotes the linear span.

\begin{lem} \label{lem:set-obda}
  Let $M$ be a subset of $\binvn{t}$ and let $e_1, \dots, e_t$ be the
  standard basis of $\binvn{t}$.
  \begin{enumerate}
    \item
      If $\dim \vspan{M} = t$ and $\abs{M}\geq t+1$, then there exists an affine permutation $T\colon\binvn{t}\to \binvn{t}$ such that
      \[
        \set{0, e_1, \dots, e_k} \subseteq T(M) \subseteq \vspan{e_1, \dots, e_k}
      \]
      and $k\in\set{t-1,t}$.
    \item
      If $\set{0, e_1, \dots, e_t} \subseteq M$ and if there is an element $m$ of $M$
      with weight $w\geq 2$, then there exists a linear permutation
      $L\colon\binvn{t}\to \binvn{t}$ such that
      \[
      \set{0, e_1, \dots, e_t, e_1+e_2+\dotsb+e_w} \subseteq L(M).
      \]
  \end{enumerate}
\end{lem}

Using \propref{prop:Sidon-equiv} and \lemref{lem:set-obda} we are able to characterise all maximal Sidon sets up to dimension 6.

\begin{prop} \label{prop:max-sidon-charac-small-dim}
  Let $M$ be a maximal Sidon set of $\binvn{t}$ and let $e_1, \dots, e_t$ be the
  standard basis of $\binvn{t}$.
  Then there exists an affine permutation $T\colon\binvn{t}\to \binvn{t}$
  such that $T(M)$ equals to
  \begin{enumerate}
    \item[t=1:] $M_1=\set{0,e_1}=\binvn{1}$ with $\abs{M_1}=2$.
    \item[t=2:] $M_2=\set{0,e_1,e_2}=\binvn{2}\setminzero$ with $\abs{M_2}=3$.
    \item[t=3:] $M_3=\set{0,e_1,e_2,e_3}$ with $\abs{M_3}=4$.
    \item[t=4:] $M_4=\set{0,e_1,e_2,e_3,e_4,e_1+e_2+e_3+e_4}$ with $\abs{M_4}=6$.
    \item[t=5:] $M_5=\set{0,e_1,e_2,e_3,e_4,e_5, e_1+e_2+e_3+e_4}$
                with $\abs{M_5}=7.$
    \item[t=6:] $M_{6a}=\set{0,e_1,e_2,e_3,e_4,e_5,e_6,e_1+e_2+e_3+e_4,e_1+e_2+e_5+e_6}$
                or\\
                $M_{6b}\, =\set{0,e_1,e_2,e_3,e_4,e_5,e_6,e_1+e_2+e_3+e_4+e_5+e_6}$\\
                with $\abs{M_{6a}} = 9 > \abs{M_{6b}} = 8$.
  \end{enumerate}
\end{prop}
\begin{proof}
  Let $M\subseteq\binvn{t}$ be maximal Sidon.
  Because of \propref{prop:Sidon-equiv} and \lemref{lem:set-obda} (a)
  we may assume without loss of generality
  that $\set{0, e_1, \dots, e_k} \subseteq M \subseteq \vspan{e_1, \dots, e_k}$ with $k\in\set{t-1,t}$.
  From $M$ maximal and \propref{prop:max-sidon-charac} follows $k=t$ and $\threeSums{M} = \binvn{t}$.
  Therefore $M$ contains no elements of weight 2 or 3.
  Hence $\binvn{t}\setminus\threeSums{M}$ consists of vectors
  of weight $\geq 4$ and the cases up to $t=4$ are shown.
  A case analysis of the maximal weight for vectors in $M$ leads, together with
  \lemref{lem:set-obda} (b) and straight forward calculations,
  to the remaining cases:
  \begin{enumerate}
    \item[t=5:]
    $M_5=\set{0,e_1,e_2,e_3,e_4,e_5, e_1+e_2+e_3+e_4}$ or\\
    $M'_5=\set{0,e_1,e_2,e_3,e_4,e_5, e_1+e_2+e_3+e_4+e_5}$\\
    with $\abs{M_5}=\abs{M'_5}=7.$
    \item[t=6:]
    $M_{6a1}=\set{0,e_1,e_2,e_3,e_4,e_5,e_6,e_1+e_2+e_3+e_4,e_{i_1}+e_{i_2}+e_5+e_6}$,\\
    $M_{6a2}=\set{0,e_1,e_2,e_3,e_4,e_5,e_6,e_1+e_2+e_3+e_4+e_5,e_{j_1}+e_{j_2}+e_{j_3}+e_6}$ or
    $M_{6b} =\set{0,e_1,e_2,e_3,e_4,e_5,e_6,e_1+e_2+e_3+e_4+e_5+e_6}$\\
    with distinct $i_1,i_2\in\set{1,2,3,4}$, pairwise distinct $j_1,j_2,j_3\in\set{1,2,3,4,5}$ \\and
    $\abs{M_{6a1}} = \abs{M_{6a2}}=9 > \abs{M_{6c}} = 8$.
  \end{enumerate}
  Now we show that some of the cases above can be transformed
  into each other via affine transformations.
  The affine transformation $T_5 = L_5 + e_5$ on $\binvn{5}$ fulfils $T_5(M'_5)=M_5$
  where the used linear transformation $L_5\colon\binvn{5}\to \binvn{5}$ is defined by:
  \begin{equation*}
    e_1\mapsto e_1+e_5 \quad
    e_2\mapsto e_2+e_5 \quad
    e_3\mapsto e_3+e_5 \quad
    e_4\mapsto e_4+e_5 \quad
    e_5\mapsto e_5.
  \end{equation*}
  Simple permutations of the standard basis  $e_1, \dots, e_t$
  result in affine
  transformations $T_{61}, T_{62}\colon\binvn{6}\to \binvn{6}$
  such that
  \begin{align*}
      T_{61}(M_{6a1}) = M_{6a}& \text{ and}\\
      T_{62}(M_{6a2}) = M'_{6a}&
      = \set{0,e_1,e_2,e_3,e_4,e_5,e_6,e_1+e_2+e_3+e_4+e_5,\\
      &\qquad e_1+e_2+e_5+e_6}.
  \end{align*}
  Extending the definition of $L_5$ by $e_6\mapsto e_6+e_5$
  leads to a linear transformation $L_6\colon\binvn{6}\to \binvn{6}$
  and results in an affine transformation $T_6 = L_6 + e_5$ on $\binvn{6}$ which fulfils $T_6(M'_{6a})=M_{6a}$.
\end{proof}

For dimension 7 and 8 we were not able to classify all maximal Sidon sets,
but determined all possible sizes of maximal Sidon sets by computer calculations.
\begin{prop} \label{prop:max-sidon-calc}
    Let $M$ be a maximal Sidon set of $\binvn{t}$.
    If $t=7$, then $\abs{M} = 12$ and if $t=8$, then $\abs{M} \in\set{15,16,18}$.
\end{prop}

Examples of maximal Sidon sets with the sizes from
\propref{prop:max-sidon-calc} can be found in
\tabref{tab:maxSidonSetExamples}. We use the standard integer representation of vectors
in $\binvn{t}$: the integer $\sum_{i=0}^{t-1} a_i2^i$ in 2-adic representation
``is'' the vector $(a_0,\ldots, a_{t-1})$.

\begin{table}[h!]
  \moreSpaceBetweenRows{1.2}
  \centering
  \begin{tabular}{@{}ccl@{}} \toprule

    $t$ & $\abs{M}$ & $M$ \\
    \midrule
    7 & 12 &
    $\set{0, 1, 2, 4, 8, 16, 32, 64, 15, 60, 101, 87}$ \\
    8 & 15 &
    $\set{0, 1, 2, 4, 8, 16, 32, 64, 128, 29, 58, 116, 135, 223, 236}$ \\
    8 & 16 &
    $\set{0, 1, 2, 4, 8, 16, 32, 64, 128, 29, 58, 116, 232, 205, 135, 222}$ \\
    8 & 18 &
    $\set{0, 1, 2, 4, 8, 16, 32, 64, 128, 29, 58, 116, 232, 205, 135, 254, 91, 171}$ \\
    \bottomrule
  \end{tabular}

  \caption{ \label{tab:maxSidonSetExamples}
    Examples of maximal Sidon sets $M$ in $\binvn{t}$
    of all possible sizes for dimension $t=7$ and $t=8$.
  }
\end{table}

\begin{rem}
  Here are some details about the computer calculations used in  \propref{prop:max-sidon-calc}:
  \begin{enumerate}
    \item
      The algorithm is based on \propref{prop:sidon-extension}:
      if $M\subseteq\binvn{t}$ is Sidon and $g\in\binvn{t}\setminus\threeSums{M}$, then $M\dotcup\set{g}$ is Sidon.
    \item
      Because of \propref{prop:Sidon-equiv} and \lemref{lem:set-obda} (a)
      we may assume without loss of generality that
      $0, e_1, \dots, e_t$ is contained
      in any maximal Sidon set in $\binvn{t}$.
    \item
      For dimension 7, the assumption from (b) is sufficient
      to complete the calculations after 12 seconds.
      As a result, we get 524160 maximal Sidon sets of size 12
      containing $0, e_1, \dots, e_7$.
    \item
      For dimension 8,  the assumption from (b) is still not sufficient to complete the calculations.
      We divided the calculation into 5 subtasks
      with the help of \lemref{lem:set-obda} (b):
      let $M$ be a maximal Sidon set containing $0, e_1, \dots, e_t$.
      Then the maximal weight of all elements of $M$ is either 4,5,6,7 or 8
      and we can assume without loss of generality,
      because of the Sidon property, that
      \begin{enumerate}[(1)]
        \item
        for the maximal weight 4, $e_1+e_2+e_3+e_4$ is contained in $M$
        and the weight of all other elements is at most 4;
        \item
        for the maximal weight 5, $e_1+e_2+e_3+e_4+e_5$ is contained in $M$
        and the weight of all other elements is at most 5;
        \item
        for the maximal weight 6, $e_1+e_2+\dotsb+e_5+e_6$ is contained in $M$
        and the weight of all other elements is at most 6;
        \item
        for the maximal weight 7, $e_1+e_2+\dotsb+e_6+e_7$ is contained in $M$
        and the weight of all other elements is at most 6;
        \item
        for the maximal weight 8, $e_1+e_2+\dotsb+e_7+e_8$ is contained in $M$
        and the weight of all other elements is at most 5.
      \end{enumerate}
      Task (5) was the most time-consuming and took about 14 days without
      parallelisation.
  \end{enumerate}
\end{rem}

We close this section with a result on the 4{-}sums of Sidon sets.
As seen before in \propref{prop:max-sidon-charac},
a maximal Sidon set can be characterised via
its 3{-}sums, i.e. a Sidon set $M$ is maximal if and only if $\threeSums{M} = \binvn{t}$. For sufficiently large  Sidon sets we obtain the following result about the
4{-}sums of Sidon sets.
It is used later in \propref{thm:sidon-covering-radius} in connection with linear codes to indicate the covering radius of the code
associated with a sum-free Sidon set.

\begin{thm} \label{thm:large-sidon-set-properties}
  Let $M$ be a Sidon set of $\binvn{t}$. If $\abs{M} > \smax{t-1}$, then
  \begin{enumerate}
    \item $\fourSums{M} = \binvn{t}$ and
    \item $\dim \vspan{M} = \dim \vspan{\twoStarSums{M}} = t$.
  \end{enumerate}
\end{thm}
\begin{proof}
  We only prove (a) as (b) is a direct consequence of it.

  Assume that there exists an $a\in\binvn{t}\setminus\fourSums{M}$.
      Let  $b\in\binvn{t}\setminus a^{\orth}$ and
      $f_b \colon \binvn{t}\to \binf$ be defined by $g\mapsto b\cdot g$, where $\cdot$ denotes the standard inner product, and $^\orth$ denotes the orthogonal space with respect to this inner
product.
      We now consider the mapping $T\colon\binvn{t}\to\binvn{t}$ defined by $g\mapsto g + f_b(g) a$.
      From $b\cdot a = 1$, it follows
      $\binvn{t} = b^{\orth} \dotcup (a + b^{\orth})$
      and
      $T(g) = \begin{cases}
                g      &\text{ for } g \in b^{\orth},\\
                g + a  &\text{ for } g = a+b^{\orth},
              \end{cases}
      $
      hence $T(\binvn{t})\subseteq b^{\orth}$.
      Assuming $T(m_1)=T(m_2)$ for distinct $m_1,m_2\in M$ would lead to
      $m_1+m_2=a$, but this contradicts $a\notin \fourSums{M} = \fourStarSums{M} \cup \twoSums{M}$. Hence $\abs{T(M)}=\abs{M}$.
      Assuming $T(m_1)+T(m_2)=T(m_3)+T(m_4)$ would lead because of the Sidon property of $M$ to $m_1+m_2+m_3+m_4=a$ but this contradicts $a\notin\fourStarSums{M}$. Hence $T(M)$ is Sidon.
      But now we found a Sidon set $T(M)\subseteq b^{\orth}$ with
      $\abs{T(M)} > \smax{t-1}$, a contradiction.
\end{proof}

Note that the opposite direction of \thmref{thm:large-sidon-set-properties} is, in general, not true. For instance,
$M=\set{0, e_1, \dots, e_t}\subseteq\binvn{t}$ is a Sidon set with
$\dim \vspan{M} = \dim \vspan{\twoStarSums{M}} = t$
but $\fourSums{M}\subsetneq \binvn{t}$ and $\abs{M}\leq \smax{t-1}$
for $t\geq 5$.

\section{Sum-free sets}
\label{sec:sum-free-sets}

In this section we introduce
sum{-}free sets, give some basic properties
and show
that it is equivalent to discuss the maximum size
of a sum{-}free Sidon set instead of a Sidon set.
Then, in the next section, we recall a one-to-one correspondence between
sum{-}free Sidon sets and
linear codes with a minimum distance greater than or equal to 5,
which gives us the possibility to translate the question
about the maximum size of a Sidon set
into a question about linear codes with certain properties.

\begin{defn}
  Let $M$ be a subset of $\binvn{t}$. $M$ is called \mathKeyword{sum-free}
  if $m_1+m_2\ne m_3$ for all $m_1,m_2,m_3\in M$.
\end{defn}

By definition, $0$ is never contained in a sum{-}free set.
We give some basic properties:
{\it
  Let $M$ be a subset of $\binvn{t}$.
  \begin{enumerate}
    \item $M$ is sum-free if and only if $\twoSums{M} \cap M = \emptyset$.
    \item If $M$ is sum-free, then $\abs{M}\leq 2^{t-1}$.
    \item $M$ is sum-free and $\abs{M} = 2^{t-1}$ if and only if
    $M=H+a$ for a hyperplane $H$ of $\binvn{t}$
    (which is a linear subspace of dimension $t-1$)
    and $a\in\binvn{t}\setminus H$.
  \end{enumerate}
}
More on sum{-}free sets can be found in the survey papers of Green and Ruzsa
\cite{Green2005} as well as Tao and Vu \cite{tao17}.

The Sidon property and the maximal Sidon property are
invariant under the action of the affine group.
This is not true, in general, for
the property of being sum{-}free.
But being sum{-}free  is still invariant under the action of the general linear group:

\begin{prop} \label{prop:sum-free-equiv}
  Let $M$ be a subset of $\binvn{t}$,
  $L\colon\binvn{t}\to \binvn{t}$ be a linear permutation
  and $a\in\binvn{t}$.
  Then
  \begin{enumerate}
    \item $M$ is sum-free if and only if $L(M)$ is sum-free.
    \item $M+a$ is sum-free if and only if
    $a\in\binvn{t}\setminus\threeSums{M} =
    \binvn{t}\setminus\bigl(\threeStarSums{M}\cup M\bigr)$.
  \end{enumerate}
\end{prop}
\begin{proof}
  (a) follows directly from the definition of sum{-}free and
  from the linearity and bijectivity of $L$.
  In order to show (b) we assume that $M+a$ is not sum{-}free.
  Hence there exist $m_1,m_2,m_3\in M$ such that $m_1 + a + m_2 + a = m_3 + a$,
  thus $m_1 + m_2 + m_3 = a$.
  But this is equivalent to $a\in\threeSums{M} =
   (\threeStarSums{M}\cup M)$ and (b) is shown.
\end{proof}

The next Proposition is about extending  sum{-}free sets
and sum{-}free Sidon sets by adding elements.

\begin{prop}\label{prop:sum-free-sidon-extension}
  Let $M$ be a subset of $\binvn{t}$ and $g\in\binvn{t}\setminus M$.
  \begin{enumerate}
    \item Let $M$ be sum-free.
         Then $M\dotcup\set{g}$ is sum-free if and only if
         $g\in\binvn{t}\setminus\twoSums{M}$.
    \item Let $M$ be sum-free Sidon.
         Then $M\dotcup\set{g}$ is sum-free Sidon if and only if
         $g\in\binvn{t}\setminus\big(\threeSums{M}\cup\twoSums{M}\big)$.
    \item Let $M$ be sum-free Sidon.
         Then $M\dotcup\set{g}$ is Sidon and not sum-free if and only if
         $g\in\twoSums{M}\setminus\threeSums{M}$.
  \end{enumerate}
\end{prop}
\begin{proof}
Let $M$ be sum-free and $g\in\binvn{t}\setminus M$.
We assume that $M\dotcup\set{g}$ is not sum{-}free.
Hence there exist $m_1,m_2\in M\dotcup\set{g}$ such that $m_1+m_2 = g$.
But this is equivalent to either $g\in\twoStarSums{M}$ or $g=0$,
hence $g\in\twoSums{M} = \twoStarSums{M} \cup \set{0}$ and (a) is shown.
Cases (b) and (c) follow from (a) and \propref{prop:sidon-extension}.
\end{proof}

The case when 0 is contained in a Sidon set is of special interest.

\begin{prop}\label{prop:sum-free-sidon-extension-0}
  Let $M$ be a subset of $\binvn{t}$.
  \begin{enumerate}
    \item If $M$ is sum-free Sidon, then
    $M\dotcup \set{0}$ is Sidon and not sum-free.
    \item If $M$ is Sidon and $0\in M$,
    then $M\setminzero$ is sum{-}free Sidon.
  \end{enumerate}
\end{prop}
\begin{proof}
  (a) is a direct consequence of \propref{prop:sum-free-sidon-extension} (c)
  due to $0\in\twoSums{M}\setminus\threeSums{M}$.

  Now we consider (b). Since every subset of a Sidon set is Sidon, it  remains to show
  that $M\setminzero$ is sum{-}free.
  We assume that $M\setminzero$ is not sum{-}free. Hence, there exist
  $m_1,m_2,m_3\in M\setminzero$ such that $m_1+m_2=m_3$.
  Since $0\notin M\setminzero$ it follows that $m_1,m_2,m_3$ are
  pairwise distinct. But then, we found $m_1,m_2,m_3,0\in M$ pairwise distinct
  such that $m_1+m_2=m_3+0$ which contradicts $M$ Sidon and (b) is shown.
\end{proof}

Therefore, the problem to find   the maximum size
of a sum{-}free Sidon set is equivalent to find the maximum size  of a Sidon set.

\begin{prop}\label{prop:smax-sfsmax}
  Let $\sfsmax{t}$ denote the maximum size of a sum{-}free Sidon set
  in $\binvn{t}$. Then
  \[
  \smax{t} = \sfsmax{t}+1.
  \]
\end{prop}

\section{Linear Codes}
\label{sec:linear-codes}

A (binary) \mathKeyword{linear code} $\code$ of \mathKeyword{length} $n$
and \mathKeyword{dimension} $k$
is a $k$ dimensional vector subspace $\code$ in $\binvn{n}$.
Such a code $\code$ is called an $[n,k]$-code and
$c\in \code$ is called a \mathKeyword{code word} of $\code$.
We consider all vectors to be row vectors.

If the \mathKeyword{minimum distance} of $\code$ is $d$,
that is the minimum number of non-zero entries
of all non-zero code words of $\code$, then $\code$ is called a $[n,k,d]$-code.

A \mathKeyword{parity check matrix} of an $[n,k]$-code $\code$
is an $(n-k)\times n$ matrix $\checkMat$
such that $\code$ equals to the kernel of $\checkMat$
i.e $\code = \sett{v\in\binvn{n}}{\checkMat\cdot \transpose{v} = 0}$.
We note that the rank of $\checkMat$ is  $n-k$.

The \mathKeyword{covering radius} of an $[n, k]$-code $\code$
with parity check matrix $\checkMat$
is the smallest integer $R$
such that every binary column vector with $n - k$ entries can be written as
the sum of at most $R$ columns of $\checkMat$.

We recall a fruitful one-to-one correspondence between additive structures and
linear codes, see \cite{CZ99}.
It translates additive properties of a subset $M$ of $\binvn{t}$
into properties of an associated code of length $\abs{M}$,
such as minimum distance or covering radius,
and vice versa.
Independently from us, \cite{nagy22} also discussed the one-to-one correspondence
with the focus set on Sidon sets.
A similar discussion is done in \cite{CCZ98} for the specific case of
graphs of APN functions, which are Sidon sets.

When formulating the correspondence we will see
that we need an ordering for the elements of $\binvn{t}$.
Therefore, for the rest of this section, we assume,
that $\binvn{n}$ is endowed with an ordering,
but all what follows is independent of this ordering.

Additionally, our purpose is to read information about a given set $M$ from its associated code.
However, this is not possible if the associated code is trivial,
that is, of dimension $0$ or $\abs{M}$.
So we formulate the correspondence in such a way
that we never obtain a trivial associated code from a given set $M$.

The \mathKeyword{one-to-one correspondence}
\begin{align*}
  \left\{ \begin{array}{c}
    \text{$M\subseteq\binvn{t}\setminzero$} \\
    \text{with $\abs{M} \geq t+1$}
  \end{array}\right\}
  &\longleftrightarrow
  \left\{ \begin{array}{c}
  \text{$[n,k,d]$-code $\code$}\\
  \text{with $n-1\geq k \geq 1$ and $d\geq 3$}
  \end{array}\right\}\\
  M \quad &\,\longmapsto \quad \code_M\\
  M_\code \quad &\,\,\reflectbox{$\longmapsto$} \qquad\code
\end{align*}
is defined as follows:

Let $M$ be a subset of $\binvn{t}\setminzero$ with $\abs{M}\geq t+1$.
We define the \mathKeyword{associated matrix} $\assoMat_M$ of $M$
as the $t\times \abs{M}$ matrix,
where the columns are the vectors of $M$, i.e
\[
\assoMat_M = (\transpose{m})_{m\in M}
\]
and the \mathKeyword{associated code} $\code_M$ of $M$
is the kernel of this matrix, i.e
\[
\code_M = \sett{v\in\binvn{\abs{M}}}{\assoMat_M\cdot \transpose{v} = 0}.
\]
If the rank of  $\assoMat_M$ is $t$,
then it is a parity check matrix of the associated code $\code_M$.

The dimension of $\code_M$ is never $0$ because of $\abs{M} \geq t+1$,
and never $\abs{M}$ because $0\notin M$.

Some basic properties of the associated codes are  the following:

{\it
Let $M$ be a subset of $\binvn{t}\setminzero$ with $\abs{M}\geq t+1$
and let $\code_M$ be its associated $[\abs{M},k,d]$-code.
Then
\begin{enumerate}
  \item $\abs{M} > t \geq 2$;
  \item $\abs{M}-1 \geq k \geq \abs{M}-t \geq 1$;
  \item $k=\abs{M}-t$ if and only if $\dim \vspan{M} = t$;
  \item $d\geq 3$, as no column is 0 and no column appears twice.
\end{enumerate}
}

The columns of a parity check matrix of an $[n,k,d]$-code $\code$
with $n-1\geq k \geq 1$ and $d\geq 3$
form a subset $M_\code$ of $\binvn{t}\setminzero$
with $t=n-k$ and $\abs{M_\code}=n \geq t+1 = n-k+1$,
which we call the \mathKeyword{associated set} of $\code$.

It should be noted that the associated set $M_\code$ of a code $\code$ is not unique,
just like the parity check matrix of a code is not unique.

As an example of the one-to-one correspondence we give the following proposition
(Proposition 2.1 of \cite{CZ99}),
with the proof appended for the convenience of the reader.

\begin{prop} \label{prop:sumfree-sidon-code}
  Let $M$ be a subset of $\binvn{t}\setminzero$ with $\abs{M}\geq t+1$
  and let $\code_M$ be its associated $[\abs{M},k,d]$-code.
  Then
  \begin{enumerate}
    \item $M$ is sum{-}free if and only if $d\geq 4$;
    \item $M$ is sum{-}free Sidon if and only if $d\geq 5$.
  \end{enumerate}
\end{prop}
\begin{proof}
  The minimum distance of an associated code is at least 3.
  \begin{enumerate}
    \item
      $M$ is sum{-}free if and only if the equation
      \[
         m_1 + m_2 + m_3 = 0
      \]
      has no solution for pairwise distinct $m_1,m_2,m_3\in M$.
      This is equivalent to: $\code_M$ has no code words of weight 3.
    \item
      $M$ is sum{-}free Sidon if and only if the system of equations
      \[
        \begin{cases}
          m_1 + m_2 + m_3 = 0\\
          m_1 + m_2 + m_3 + m_4 = 0\\
        \end{cases}
      \]
      has no solution for pairwise distinct $m_1,m_2,m_3,m_4\in M$.
      This is equivalent to:
      $\code_M$ has no code words of weight 3 or 4.
  \end{enumerate}
\end{proof}
Part (a) of \propref{prop:sumfree-sidon-code} is also included in \cite{CP92}.
For consequences in the APN setting, see \cite{CCZ98}.

The following Theorem gives details on the one-to-one correspondence with a focus
on Sidon sets.

\begin{thm} \label{thm:max-min-dist}
  Let $M$ be a subset of $\binvn{t}\setminzero$ with $\abs{M}\geq t+1$ and
  let $\code_M$ be its associated $[\abs{M},k,d]$-code.
  \begin{enumerate}
    \item If $M$ is sum{-}free and if $\abs{M}\geq\smax{t}$, then $d=4$ and $M$ is not Sidon.
    \item If $M$ is sum{-}free Sidon and if $\abs{M}\geq\smax{t-1}$, then $k=\abs{M}-t$ and $\assoMat_M$ is a parity check matrix of $\code_M$.
    \item If $M$ is sum{-}free Sidon and if $\abs{M}\geq\smax{t-1}+1$, then $d = 5$.
  \end{enumerate}
\end{thm}
\begin{proof}
  \begin{enumerate}
    \item Let $M\subseteq\binvn{t}\setminzero$ be sum{-}free.
          Then $d\geq 4$ from \propref{prop:sumfree-sidon-code} (a).
          If $\abs{M}\geq\smax{t}$ then $\abs{M\dotcup\set{0}} > \smax{t}$ and
          therefore  neither $M\dotcup\set{0}$ nor $M$ is Sidon,
          hence $d=4$ due to \propref{prop:sumfree-sidon-code} (b).

    \item If $M\subseteq\binvn{t}\setminzero$ is sum{-}free Sidon
          and if $\abs{M}\geq\smax{t-1}$ then
          $M\dotcup\set{0}$ is still Sidon and $\abs{M\dotcup\set{0}}>\smax{t-1}$.
          From \thmref{thm:large-sidon-set-properties} (b) follows
          $\dim\vspan{M}=t$ and therefore $k = \abs{M}-t$.

    \item From $M$ sum{-}free Sidon follows that $d\geq 5$
          and due to $\abs{M}\geq\smax{t-1}+1$ and (a), the matrix
 $\checkMat_M :=\assoMat_M$ is a parity check matrix of $\code_M$.

          Assume that $d\geq 6$.
          Then $\punct{\code_M}$,
          the puncturing of $\code_M$ (remove one column and one row of $\checkMat_M$),
          is an $[\abs{M}-1,\abs{M}-t,d']$-code with $d'\geq 5$.
          Therefore the columns of the check matrix
          $\checkMat_{\punct{\code_M}}$ of $\punct{\code_M}$
          form a sum{-}free Sidon set $M'\subseteq\binvn{t'}$
          with $\abs{M'} = \abs{M}-1$ and $t' = t-1$.
          From $\abs{M}\geq\smax{t-1}+1$ follows that $\abs{M'}= \abs{M}-1 \geq\smax{t-1}$
          and $M'\subseteq\binvn{t-1}$ not Sidon due to (a).
  \end{enumerate}
\end{proof}

Another interesting connection between a code property
 and a Sidon property is the following:
\begin{thm} \label{thm:sidon-covering-radius}
  Let $M$ be a subset of $\binvn{t}\setminzero$ with $\abs{M}\geq t+1$ and
  let $\code_M$ be its associated $[\abs{M},k,d]$-code with covering radius $R$.
  \begin{enumerate}
    \item If $M$ is sum{-}free Sidon and $\abs{M}\geq\smax{t-1}$, then
    $R=3$ or $R=4$.
    \item $M$ is maximal sum{-}free Sidon (that means we cannot
    extend it to a larger sum{-}free Sidon set by adding elements) if and only if $R=3$.
  \end{enumerate}
\end{thm}
\begin{proof}
  \begin{enumerate}
    \item
      $M$ is Sidon and therefore $\abs{\twoStarSums{M}} = \binom{\abs{M}}{2}$.
      From $\abs{M}\geq t+1$ follows that $\abs{M}>t\geq 2$.
      But then $\abs{\twoStarSums{M}} = \binom{\abs{M}}{2} < 2^t$ and $R\geq 3$.
      From \thmref{thm:large-sidon-set-properties} (a) follows $R \leq 4$.
    \item
      This is a direct consequence of \propref{prop:max-sidon-charac}.
  \end{enumerate}
\end{proof}

We close this section by discussing the best possible minimum distance of
a code with given length $n$ and dimension $k$.
Therefore we define
\[
\dmax{n,k} = \max\sett{d}{\text{there exists an $[n,k,d]$-code}}.
\]
as the \mathKeyword{maximal minimum distance} of a code with given
length $n$ and dimension $k$.
It is one of the main properties of \mathKeyword{optimal} codes
and frequently listed as a matrix $\dmaxTab$, for example in
Grassl's codes table \cite{Grassl:codetables} (\href{http://codetables.de}{http://codetables.de})
or the codes table of the MinT project from Schürer and Schmid \cite{schurer2006mint} (\url{http://mint.sbg.ac.at/}).
Now we translate \thmref{thm:max-min-dist} to some properties
of the subdiagonals of $\dmaxTab$, namely the entries $(\dmax{n,n-t})_n$
for a fixed $t$.

\begin{prop} \label{prop:max-min-charac}
  Let be $n,t\in\nat$ with $n > t \geq 2$. Then
  \begin{enumerate}
    \item $\dmax{n, n-t}=3$ if and only if $2^{t-1} < n < 2^{t}$;
    \item $\dmax{n, n-t}=4$ if and only if $\smax{t} \leq n \leq 2^{t-1}$;
    \item $\dmax{n, n-t}=5$ if and only if $\smax{t-1} < n < \smax{t}$;
    \item $\dmax{n, n-t}\geq 6$ if and only if $n \leq \smax{t-1}$.
  \end{enumerate}
\end{prop}
\begin{proof}
  From our correspondence it follows that
  every $M\subseteq\binvn{t}\setminzero$ with $\abs{M}\geq t+1$
  gives rise to an associated $[\abs{M},k,d]$-code with $d\geq 3$.
  \begin{enumerate}
    \item
      If $\abs{M}>2^{t-1}$, then $\dim\vspan{M}=t$ and $k=\abs{M}-d$,
      but $M$ cannot be sum{-}free anymore.
      Thus $d=3$ and (a) is shown.

    \item
      Let  $M=H+a$ with a hyperplane $H$ of $\binvn{t}$
      (which is a linear subspace of dimension $t-1$) and $a\in\binvn{t}\setminus H$.
      Hence $M$ is sum-free and $d=4$ from \thmref{thm:max-min-dist} (a).
      Additionally, $\dim\vspan{M}=t$ and $k=\abs{M}-d$.
      Now, removing elements from $M$ such that $\dim\vspan{M}=t$ is still valid,
      leads, together with \thmref{thm:max-min-dist} (a), to (b).

    \item
      This is a direct consequence of \thmref{thm:max-min-dist} (b) and (c).

    \item
      This follows from (a), (b) and (c).
  \end{enumerate}

\end{proof}

\section{Non-existence results}
\label{sec:non-existence-results}

Due to the importance of non-existence statements for Sidon sets and as well for linear codes we reformulate \propref{prop:max-min-charac}.

\begin{cor} \label{cor:non-exist-equi}
  Let be $n,t\in\nat$ with $n > t \geq 2$. Then the following statements are equivalent:
  \begin{enumerate}
    \item There is no $[n,n-t,5]$ code.
    \item There is no Sidon set $M\subseteq \binvn{t}$ of size $n+1$.
    \item $\smax{t} \leq n$.
  \end{enumerate}
\end{cor}

The following result from Brouwer and Tolhuizen \cite{BT93}
is achieved by a sharpening of the Johnson bound.
It improves the trivial upper bound \equRefTrivialBound{} for odd dimension.
\begin{thm} [\cite{BT93}] \label{thm:BT93}
  There is no $\left[n, n - t, 5\right]$ code for $n = 2^{(t+1)/2} - 2$,
  hence
  \[
  \smax{t} \leq 2^{(t+1)/2} - 2
  \]
  for $t$ odd with $t \geq 7$.
\end{thm}

With arguments similar to those used by Brouwer and Tolhuizen,
we are able to generalise this result to arbitrary $t\geq 6$
and thereby further improve the trivial upper bound \equRefTrivialBound{}.
This improves also a recent bound given by Tait and Won (Theorem 5.1 of \cite{TaitWon21}).

\begin{thm} \label{thm:improved-codes-non-existence-sidon-bound}
  Let  $t \geq 6$, and write  $\floor{\sqrt{2^{t+1}} + 0.5} - 4 = 3a+b$ with
  $a\in\inte_{\geq 0}$, $b\in\set{0,1,2}$,
  $\eps = {\sqrt{2^{t+1}} + 0.5} - \floor{\sqrt{2^{t+1}} + 0.5} \in[0,1)$
  and
  \[
    \lambda_{a,b,\eps} =
      \begin{cases}
        1   &\text{ for $a$ odd and $b=0$},\\
        2   &\text{ for $a$ odd, $b=1$ and
                                $0\leq\eps \leq 1-\frac{1}{2^{(t-4)/2}}$},\\
        1   &\text{ for $a$ odd, $b=1$ and
                                $1-\frac{1}{2^{(t-4)/2}} < \eps<1$},\\
        2   &\text{ for $a$ odd and $b=2$},\\
        2   &\text{ for $a$ even, $b=0$  and $0\leq\eps \leq 0.5$},\\
        1   &\text{ for $a$ even, $b=0$ and $0.5<\eps<1$},\\
        2   &\text{ for $a$ even, $b=1$ and
                                  $0\leq\eps \leq 1-\frac{1}{2^{(t-5)/2}}$},\\
        1   &\text{ for $a$ even, $b=1$ and
                                  $1-\frac{1}{2^{(t-5)/2}} < \eps \leq 1-\frac{1}{2^{(t+7)/2}}$},\\
        0   &\text{ for $a$ even, $b=1$ and
                                  $1-\frac{1}{2^{(t+7)/2}} < \eps<1$},\\
        0   &\text{ for $a$ even and $b=2$}.
      \end{cases}
  \]

  Then there is no $\left[n_t, n_t - t, 5\right]$ code for
  $n_t = \floor{\sqrt{2^{t+1}} + 0.5} - \lambda_{a,b,\eps}$
  and therefore
  \begin{equation} 
    \smax{t} \leq
      \begin{cases}
        2^{\frac{t+1}{2}} - 2           &\text{ for $t$ odd},\\
        \floor{\sqrt{2^{t+1}} + 0.5} - \lambda_{a,b,\eps} &\text{ for $t$ even}.
      \end{cases}
  \end{equation}
\end{thm}
\begin{proof}
  Because of \corref{cor:non-exist-equi} it is sufficient to show
  the non-existence of an $\left[n_t, n_t - t, 5\right]$ code.

  Let us recall some arguments from the proof of
  \thmref{thm:BT93} in \cite{BT93}.
  Let $C$ be an arbitrary (linear or non-linear) code
  of length $n$, with minimum distance $5$, and
  where, on the average, each codeword is at distance 5
  from $a_5$ other codewords.
  The Johnson upper bound (Theorem 1 of \cite{Johnson62}) states
  \begin{equation} 
    \abs{C} \leq \frac{2^n}{s}
  \end{equation}
  for
  \begin{equation*}
    s = 1 + n + \frac{n(n-1)}{2} + \frac{1}{\floor{n/3}}\Bigl(\binom{n}{3}-10\cdot a_5\Bigr)
  \end{equation*}
  and the term $10\cdot a_5$ can be estimated by
  \[
    10\cdot a_5 \leq \floor{\frac{n-2}{3}}\cdot\binom{n}{2}.
  \]
  Brouwer and Tolhuizen sharpened the estimate of $10\cdot a_5$
  for linear codes in the following way.
  Let $C$ be additionally linear and
  let $n-2=3\cdot a+b$ with $a\in\inte_{\geq 0}$, $b\in\set{0,1,2}$, then
  \begin{equation}
    10\cdot a_5 \leq
      \begin{cases}
        a\cdot\binom{n}{2}
            &\text{ for $a$ odd and $b=0$,}\\
        (a-1)\cdot\binom{n}{2}
            &\text{ for $a$ odd and $b\in\set{1,2}$,}\\
        (a-1)\cdot\binom{n}{2} +
            \frac{\binom{n}{b}\cdot\binom{b+2}{2}}{\binom{b+2}{b}}
            &\text{ for $a$ even.}
      \end{cases}
  \end{equation}
  Our purpose is to show that if $C$ is an $\left[n_t, n_t - t, 5\right]$ code,
  then $2s>2^{t+1}$ or equivalently $s>2^t$.
  But this contradicts $s\leq 2^t$ which follows from \equRefContradict{}.

  Let us therefore distinguish 6 cases depending on wether $a$
  is odd or even and wether $b$ equals 2, 1 or 0.\\[12px]
  \noindent\textbf{Case $a$ odd:}
  If $\mathbf{b=2}$, then $a-1 = \frac{n-7}{3}$.
  From \equRefEstimate{} follows
  \[
    10\cdot a_5 \leq (a-1)\cdot\binom{n}{2}
    = \frac{n(n-1)(n-7)}{6}
  \]
  and therefore
  \[
    s \geq 1 + n + \frac{n(n-1)}{2} + \frac{5}{2}n
  \]
  which is equivalent to
  \begin{equation} 
    2s \geq (n+3)^2 - 7.
  \end{equation}
  Now we set $\lambda_{a,b,\eps}=2$,
  $n=n_t=\floor{\sqrt{2^{t+1}} + 0.5} - 2 = \sqrt{2^{t+1}} - \frac{3}{2} - \eps$
  with $\eps\in[0,1)$ and $k=n-t$.
  Hence
  \begin{align*}
    2s \geq (n_t+3)^2 - 7 &= (\sqrt{2^{t+1}} + \frac{1}{2} +1 - \eps)^2 - 7\\
      &> 2^{t+1} + \sqrt{2^{t+1}} +\frac{1}{4} -7\\
      &> 2^{t+1}
  \end{align*}
  when $t\geq 5$, and this contradicts \equRefContradict{}.

  If $\mathbf{b=1}$, then $a-1 = \frac{n-6}{3}$.
  From \equRefEstimate{} follows
  \[
    10\cdot a_5 \leq (a-1)\cdot\binom{n}{2}
    = \frac{n(n-1)(n-6)}{6}
  \]
  and therefore
  \[
    s \geq 1 + n + \frac{n(n-1)}{2} + 2(n-1)
  \]
  which is equivalent to
  \begin{equation} 
    2s \geq (n+\frac{5}{2})^2 - \frac{33}{4}.
  \end{equation}
  Now we set $\lambda_{a,b,\eps}=2$,
  $n=n_t=\floor{\sqrt{2^{t+1}} + 0.5} - 2 = \sqrt{2^{t+1}} - \frac{3}{2} - \eps$
  with $\eps\in[0,1)$, $k=n-t+1$ and want to show that
  \[
    2s \geq (\sqrt{2^{t+1}} + 1 - \eps)^2 - \frac{33}{4} > 2^{t+1}.
  \]
  This is true if $0 \leq \eps \leq 1-\frac{1}{2^{(t-4)/2}}$ and
  $t\geq 5$ because setting $\eps = 1-\frac{1}{2^{(t-4)/2}}$ leads to
  \begin{align*}
    2s \geq (\sqrt{2^{t+1}} + \frac{1}{2^{(t-4)/2}})^2 - \frac{33}{4}
      &= 2^{t+1} + \frac{2\sqrt{2^{t+1}}}{2^{(t-4)/2}} +\frac{1}{2^{t-4}} - \frac{33}{4}\\
      &= 2^{t+1} + 8\sqrt{2} +\frac{1}{2^{t-4}} - 2\frac{33}{4}\\
      &> 2^{t+1}
  \end{align*}
  which contradicts \equRefContradict{}.
  If $1-\frac{1}{2^{(t-4)/2}} < \eps < 1$ we set
  $\lambda_{a,b,\eps}=1$, $n = n_t = \floor{\sqrt{2^{t+1}} + 0.5} - 1$,
  $k=n-t$ and are now in the case
  $a$ odd and $b=2$. Putting these values into \equRefAoddBtwo{} leads to
  \[
    2s \geq (n_t+3)^2 - 7 > 2^{t+1}
  \]
  which contradicts \equRefContradict{}.

  If $\mathbf{b=0}$, then $a = \frac{n-2}{3}$.
  From \equRefEstimate{} follows
  \[
    10\cdot a_5 \leq a\cdot\binom{n}{2}
    = \frac{n(n-1)(n-2)}{6}
  \]
  and therefore
  \[
    s \geq 1 + n + \frac{n(n-1)}{2}
  \]
  which is equivalent to
  \begin{equation} 
    2s \geq (n + \frac{1}{2})^2 + \frac{7}{4}.
  \end{equation}
  But setting $\lambda_{a,b,\eps}=2$, $n = \floor{\sqrt{2^{t+1}} + 0.5} - 2$
  and $k=n-t$ does not contradict \equRefContradict{}.
  Therefore we set $\lambda_{a,b,\eps}=1$,
  $n = n_t = \floor{\sqrt{2^{t+1}} + 0.5} - 1$, $k=n_t-t$
  and are now in the case $a$ odd and $b=1$.
  Putting these values into \equRefAoddBone{} leads to
  \[
    2s \geq (n_t+\frac{5}{2})^2 - \frac{33}{4} > 2^{t+1}
  \]
  for $t\geq 3$ which contradicts \equRefContradict{}.

  \noindent\textbf{Case $a$ even:}
  If $\mathbf{b=2}$, then $a = \frac{n-4}{3}$. From \equRefEstimate{} follows
  \[
    10\cdot a_5 \leq \frac{n(n-1)(n-4)}{6}
  \]
  and therefore
  \begin{align*}
    s &\geq 1 + n + \frac{n(n-1)}{2} + \frac{3}{n-1}
            \left( \binom{n}{3} - \frac{n(n-1)(n-4)}{6} \right)\\
      &= 1 + n + \frac{n(n-1)}{2} + n
  \end{align*}
  which is equivalent to
  \begin{equation} 
    2s \geq (n+\frac{3}{2})^2 - \frac{1}{4}.
  \end{equation}
  But setting $\lambda_{a,b,\eps}=2$, $n = \floor{\sqrt{2^{t+1}} + 0.5} - 2$
  and $k=n-t$ does not contradict \equRefContradict{}.
  Therefore we set $\lambda_{a,b,\eps}=1$,
  $n = n_t = \floor{\sqrt{2^{t+1}} + 0.5} - 1$, $k=n_t-t$
  and are now in the case $a$ odd and $b=0$.
  But again, putting these values into \equRefAoddBzero{} does not contradict \equRefContradict{}.
  Hence we set $\lambda_{a,b,\eps}=0$,
  $n = n_t = \floor{\sqrt{2^{t+1}} + 0.5}$, $k=n-t$
  and are now in the case $a$ odd and $b=1$.
  Now, putting these values into \equRefAoddBone{} contradicts \equRefContradict{}.

  If $\mathbf{b=1}$, then $a = \frac{n-3}{3}= \frac{n}{3}-1$. From \equRefEstimate{} follows
  \[
    10\cdot a_5 \leq \frac{n(n-1)(n-6)}{6} + n = \frac{n(n-3)(n-4)}{6}
  \]
  and therefore
  \begin{align*}
    s &\geq 1 + n + \frac{n(n-1)}{2} + \frac{3}{n}
            \left( \binom{n}{3} - \frac{n(n-3)(n-4)}{6} \right)\\
      &= 1 + n + \frac{n(n-1)}{2} + 2n-5
  \end{align*}
  which is equivalent to
  \begin{equation} 
    2s \geq (n+\frac{5}{2})^2 - \frac{57}{4}.
  \end{equation}
  Now we set $\lambda_{a,b,\eps}=2$,
  $n=n_t=\floor{\sqrt{2^{t+1}} + 0.5} - 2 = \sqrt{2^{t+1}} + - \frac{3}{2} - \eps$
  with $\eps\in[0,1)$, $k=n-t$ and want to show that
  \[
    2s \geq (\sqrt{2^{t+1}} + 1 - \eps)^2 - \frac{57}{4} > 2^{t+1}.
  \]
  This is true if $0 \leq \eps \leq 1-\frac{1}{2^{(t-5)/2}}$
  and $t\geq 6$ because setting $\eps = 1-\frac{1}{2^{(t-5)/2}}$ leads to
  \begin{align*}
    2s \geq (n_t+\frac{5}{2})^2 - \frac{57}{4}
      &= (\sqrt{2^{t+1}} + \frac{1}{2^{(t-5)/2}})^2 - \frac{57}{4}\\
      &= 2^{t+1} + \frac{2\sqrt{2^{t+1}}}{2^{(t-5)/2}} +\frac{1}{2^{t-5}} - \frac{57}{4}\\
      &= 2^{t+1} + 16 +\frac{1}{2^{t-5}} - \frac{57}{4}\\
      &> 2^{t+1}
  \end{align*}
  which contradicts \equRefContradict{}.
  If $1-\frac{1}{2^{(t-5)/2}} < \eps < 1$ we try setting
  $\lambda_{a,b,\eps}=1$, $n = n_t = \floor{\sqrt{2^{t+1}} + 0.5} - 1$,
  $k=n-t$ and are now in the case $a$ even and $b=2$.
  Putting these values into \equRefAevenBtwo{} we want to show that
  \[
    2s \geq (n_t+\frac{3}{2})^2 - \frac{1}{4} > 2^{t+1}.
  \]
  This is true if $\eps \leq 1-\frac{1}{2^{(t+7)/2}}$
  because setting $\eps = 1-\frac{1}{2^{(t+7)/2}}$
  leads to
    \begin{align*}
    2s \geq (n_t+\frac{3}{2})^2 - \frac{1}{4}
      &= (\sqrt{2^{t+1}} + \frac{1}{2^{(t+7)/2}})^2 - \frac{1}{4}\\
      &= 2^{t+1} + \frac{2\sqrt{2^{t+1}}}{2^{(t+7)/2}} +\frac{1}{2^{t+7}} - \frac{1}{4}\\
      &= 2^{t+1} + \frac{1}{4} +\frac{1}{2^{t+7}} - \frac{1}{4}\\
      &> 2^{t+1}
  \end{align*}
  which contradicts \equRefContradict{}.
  Therefore we set $\lambda_{a,b,\eps}=0$,
  $n = n_t = \floor{\sqrt{2^{t+1}} + 0.5}$, $k=n-t$
  if $1-\frac{1}{2^{(t+7)/2}} < \eps < 1$
  and are now in the case $a$ odd and $b=0$.
  Putting these values into \equRefAoddBzero{} contradicts \equRefContradict{}.

  If $\mathbf{b=0}$, then $a = \frac{n-2}{3}$. From \equRefEstimate{} follows
  \[
    10\cdot a_5 \leq \frac{n(n-1)(n-5)}{6} + 1
  \]
  and therefore
  \begin{align*}
    s &\geq 1 + n + \frac{n(n-1)}{2} + \frac{3}{n-2}
            \left( \binom{n}{3} - \frac{n(n-1)(n-5)}{6} - 1 \right)\\
      &= 1 + n + \frac{n(n-1)}{2} + \frac{3n+3}{2}
  \end{align*}
  which is equivalent to
  \[
    2s \geq (n+2)^2 + 1.
  \]
  But setting $\lambda_{a,b,\eps}=2$, $n = n_t =\floor{\sqrt{2^{t+1}} + 0.5} - 2$
  and $k=n-t$ only contradicts \equRefContradict{} if $0 \leq \eps \leq \frac{1}{2}$.
  If $\frac{1}{2} < \eps <1$ we set $n = n_t =\floor{\sqrt{2^{t+1}} + 0.5} - 1$,
  $k=n_t-t$ and are now in the case $a$ even and $b=1$.
  Putting these values into \equRefAevenBone{} leads to
  \[
    2s \geq (n_t+\frac{5}{2})^2 - \frac{57}{4} > 2^{t+1}
  \]
  for $t\geq 5$ which contradicts \equRefContradict{}.
\end{proof}

On the codes side, \thmref{thm:improved-codes-non-existence-sidon-bound}
improves for $t$ even and $t\geq 16$ several entries in the codes table
of the MinT project \cite{schurer2006mint} (\url{http://mint.sbg.ac.at/}).
Some examples are listed in \corref{cor:improvedMaxMinDist}. In \cite{schurer2006mint},
the maximal minimum distance was listed as 4 or 5, but now we know that it is 4:

\begin{cor} \label{cor:improvedMaxMinDist}
  The maximal minimum distance of a linear code with the following parameters
  $[n,k]$ is 4:
  \begin{table}[ht]
\centering
\moreSpaceBetweenRows{1.3}
\begin{tabular}{lll}
\href{http://mint.sbg.ac.at/query.php?i=c&var=q-T-\%CE\%BB-t-d-m-n-k&b=2&s=360&n=344&de=1&p=snd}{$[360,344]$} &
\href{http://mint.sbg.ac.at/query.php?i=c&var=q-T-\%CE\%BB-t-d-m-n-k&b=2&s=723&n=705&de=1&p=snd}{$[723,705]$} &
\href{http://mint.sbg.ac.at/query.php?i=c&var=q-T-\%CE\%BB-t-d-m-n-k&b=2&s=1446&n=1426&de=1&p=snd}{$[1446,1426]$} \\
\href{http://mint.sbg.ac.at/query.php?i=c&var=q-T-\%CE\%BB-t-d-m-n-k&b=2&s=2895&n=2873&de=1&p=snd}{$[2895,2873]$} &
\href{http://mint.sbg.ac.at/query.php?i=c&var=q-T-\%CE\%BB-t-d-m-n-k&b=2&s=5791&n=5767&de=1&p=snd}{$[5791, 5767]$},
\href{http://mint.sbg.ac.at/query.php?i=c&var=q-T-\%CE\%BB-t-d-m-n-k&b=2&s=5792&n=5768&de=1&p=snd}{$[5792, 5768]$} &
\href{http://mint.sbg.ac.at/query.php?i=c&var=q-T-\%CE\%BB-t-d-m-n-k&b=2&s=11583&n=11557&de=1&p=snd}{$[11583, 11557]$}\\
\end{tabular}
\end{table}
\end{cor}

\begin{proof}
  For $t\in\set{16,18,20,22,24,26}$ we follow
  \thmref{thm:improved-codes-non-existence-sidon-bound},
  calculate $a,b$ and $\eps$,
  and obtain $\lambda_{a,b,\eps}$, n and k as in the following table:
  \begin{center}
  \moreSpaceBetweenRows{1.2}
  \begin{tabular}{@{}ccccccc@{}} \toprule
    $t$ & $\floor{\sqrt{2^{t+1}} + 0.5}$ & $a$ & $b$ & $\eps$ & $\lambda_{a,b,\eps}$ & $[n,k]$\\
    \midrule
    16 &   362 &  119 & 1 & $\approx 0.538 < 0.984$ & 2 &
    \href{http://mint.sbg.ac.at/query.php?i=c&var=q-T-\%CE\%BB-t-d-m-n-k&b=2&s=360&n=344&de=1&p=snd}{$[360,344]$}\\

    18 &   724 &  240 & 0 & $\approx 0.577 > 0.500$ & 1 &
    \href{http://mint.sbg.ac.at/query.php?i=c&var=q-T-\%CE\%BB-t-d-m-n-k&b=2&s=723&n=705&de=1&p=snd}{$[723,705]$}\\

    20 &  1448 &  481 & 1 & $\approx 0.654 < 0.996$ & 2 &
    \href{http://mint.sbg.ac.at/query.php?i=c&var=q-T-\%CE\%BB-t-d-m-n-k&b=2&s=1446&n=1426&de=1&p=snd}{$[1446,1426]$} \\

    22 &  2896 &  964 & 0 & $\approx 0.809 > 0.500$   & 1 &
    \href{http://mint.sbg.ac.at/query.php?i=c&var=q-T-\%CE\%BB-t-d-m-n-k&b=2&s=2895&n=2873&de=1&p=snd}{$[2895,2873]$} \\

    24 &  5793 & 1929 & 2 & $\approx 0.118 < 0.999$ & 2 &
    \href{http://mint.sbg.ac.at/query.php?i=c&var=q-T-\%CE\%BB-t-d-m-n-k&b=2&s=5791&n=5767&de=1&p=snd}{$[5791, 5767]$} and\\
     & & & & & &
    \href{http://mint.sbg.ac.at/query.php?i=c&var=q-T-\%CE\%BB-t-d-m-n-k&b=2&s=5792&n=5768&de=1&p=snd}{$[5792, 5768]$}\\

    26 & 11585 & 3860 & 1 & $\approx 0.737 < 0.999$ & 2 &
    \href{http://mint.sbg.ac.at/query.php?i=c&var=q-T-\%CE\%BB-t-d-m-n-k&b=2&s=11583&n=11557&de=1&p=snd}{$[11583, 11557]$}\\
    \bottomrule
  \end{tabular}
  \end{center}

\end{proof}

\section{Conclusion}

We finish by giving \tabref{tab:maxSExt} about the maximum size of Sidon sets
and related bounds/constructions in small dimensions.
The codes bound mentioned in this table arises from \propref{prop:max-min-charac} (b)
and Grassl's codes table \cite{Grassl:codetables} (\href{http://codetables.de}{http://codetables.de}).
Similarly, the codes constructions come from \propref{prop:max-min-charac} (c)
and Grassl's codes table.

For example, take column $t=12$ from \tabref{tab:maxSExt}:
The calculation of the trivial bound \equRefTrivialBound{} leads to $\smax{t}\leq 91$
and that of the new bound \equRefNewBound{}
from \thmref{thm:improved-codes-non-existence-sidon-bound}
leads to $\smax{t}\leq 90$ ($a=29$, $b=0$, $\eps\approx 0.009$ and $\lambda_{a,b,\eps}=1$).

\propref{prop:max-min-charac} (b) gives us the codes bound,
that is the smallest $n$, such that $\dmax{n, n-12}=4$.
A look at Grassl's codes table leads to $\smax{t}\leq n=89$,
since \href{http://codetables.de/BKLC/BKLC.php?q=2&n=89&k=77}{$\dmax{89, 77}=4$}
but \href{http://codetables.de/BKLC/BKLC.php?q=2&n=88&k=76}{$\dmax{88, 76}=4\mbox{\ or }5$}.

The codes construction uses \propref{prop:max-min-charac} (c) in the following way:
Finding the largest $n$ such that $\dmax{n, n-12}\geq 5$
gives a sum-free Sidon set,
and adding 0 leads to $n+1$,
which is the size of the largest known Sidon set.
Again, Grassl's codes table leads to $n=65$ and so $\smax{t}\geq 66$,
since \href{http://codetables.de/BKLC/BKLC.php?q=2&n=65&k=53}{$\dmax{65, 53}=5$}
but \href{http://codetables.de/BKLC/BKLC.php?q=2&n=66&k=54}{$\dmax{66, 54}=4\mbox{\ or }5$}.

\begin{table}[ht]
  \moreSpaceBetweenRows{1.2}
  \centering
  \begin{tabular}{@{}rllllllllllll@{}}\toprule
    $t$ &
    4 & 5 & 6  & 7  & 8  & 9  & 10 & 11 & 12 & 13  & 14  & 15  \\
    \midrule
    Trivial bound \equRefTrivialBound{}  &
    6 & 8 & 11 & 16 & 23 & 32 & 45 & 64 & 91 & 128 & 181 & 256 \\
    New bound \equRefNewBound{} &
    &   & 10 & 14 & 21 & 30 & 43 & 62 & 90 & 126 & 180 & 254 \\
    Codes bound &
    \href{http://codetables.de/BKLC/BKLC.php?q=2&n=6&k=2}{6} & 
    \href{http://codetables.de/BKLC/BKLC.php?q=2&n=7&k=2}{7} & 
    \href{http://codetables.de/BKLC/BKLC.php?q=2&n=9&k=3}{9} & 
    \href{http://codetables.de/BKLC/BKLC.php?q=2&n=12&k=5}{12} & 
    \href{http://codetables.de/BKLC/BKLC.php?q=2&n=18&k=10}{18} & 
    \href{http://codetables.de/BKLC/BKLC.php?q=2&n=24&k=15}{24} & 
    \href{http://codetables.de/BKLC/BKLC.php?q=2&n=34&k=24}{34} & 
    \href{http://codetables.de/BKLC/BKLC.php?q=2&n=58&k=47}{58} & 
    \href{http://codetables.de/BKLC/BKLC.php?q=2&n=89&k=77}{89} & 
    \href{http://codetables.de/BKLC/BKLC.php?q=2&n=125&k=112}{125} & 
    \href{http://codetables.de/BKLC/BKLC.php?q=2&n=179&k=165}{179} & 
    \href{http://codetables.de/BKLC/BKLC.php?q=2&n=254&k=239}{254} \\ 

    $\smax{t}$ &
    6 & 7 &  9 & 12 & 18 & 24 & 34 & ?  & ?  &  ?  & ?   &  ?  \\
    Codes constr. &
    \href{http://codetables.de/BKLC/BKLC.php?q=2&n=5&k=1}{6} & 
    \href{http://codetables.de/BKLC/BKLC.php?q=2&n=6&k=1}{7} & 
    \href{http://codetables.de/BKLC/BKLC.php?q=2&n=8&k=2}{9} & 
    \href{http://codetables.de/BKLC/BKLC.php?q=2&n=11&k=4}{12} & 
    \href{http://codetables.de/BKLC/BKLC.php?q=2&n=17&k=9}{18} & 
    \href{http://codetables.de/BKLC/BKLC.php?q=2&n=23&k=14}{24} & 
    \href{http://codetables.de/BKLC/BKLC.php?q=2&n=33&k=23}{34} & 
    \href{http://codetables.de/BKLC/BKLC.php?q=2&n=47&k=36}{48} & 
    \href{http://codetables.de/BKLC/BKLC.php?q=2&n=65&k=53}{66} & 
    \href{http://codetables.de/BKLC/BKLC.php?q=2&n=81&k=68}{82} & 
    \href{http://codetables.de/BKLC/BKLC.php?q=2&n=128&k=114}{129} & 
    \href{http://codetables.de/BKLC/BKLC.php?q=2&n=151&k=136}{152} \\ 
    \bottomrule
  \end{tabular}
  \caption{ \label{tab:maxSExt}
    Maximal size of a Sidon set in $\binvn{t}$ and related bounds/constructions.
  }
\end{table}

\printbibliography

\end{document}